\documentclass{amsart} 
 
\usepackage{amsmath, amsthm, amssymb}
\usepackage{mathrsfs, bm, txfonts}
\usepackage{hyperref}

\newcommand{\al}{\alpha}
\newcommand{\be}{\beta}
\newcommand{\de}{\delta}

\newcommand{\ga}{\gamma}
\newcommand{\ka}{\kappa}
\newcommand{\la}{\lambda}
\newcommand{\om}{\omega}
\newcommand{\si}{\sigma}
\newcommand{\te}{\theta}

\newcommand{\ze}{\zeta}
\newcommand{\De}{\Delta}
\newcommand{\Ga}{\Gamma}

\newcommand{\Om}{\Omega}

\newcommand{\RR}{\mathbb{R}}

\renewcommand{\SS}{\mathbb{S}}

\newcommand{\gb}{\bar{g}}
\newcommand{\Gab}{\bar{\Gamma}}
\newcommand{\gad}{\dot{\gamma}}

\newcommand{\gz}{\mathring{g}}
\newcommand{\nablaz}{\mathring{\nabla}}
\newcommand{\Tz}{\mathring{T}}

\newcommand{\hz}{\mathring{h}}

\newcommand{\zez}{\mathring{\zeta}}

\newcommand{\Scal}{\mathcal{S}}
\newcommand{\Wcal}{\mathcal{W}}
\newcommand{\nablab}{\bar{\nabla}}

\newcommand{\dmu}[1]{\,d\mu_{#1}}

\newcommand{\pa}{\partial}

\newcommand{\pat}[1]{\frac{\pa #1}{\pa t}}

\newtheorem{theorem}{Theorem}[section]
\newtheorem{lemma}[theorem]{Lemma}

\newtheorem{corollary}[theorem]{Corollary}

\newtheorem{remark}[theorem]{Remark}

\title{Stability of geodesic spheres in $\mathbb{S}^{n+1}$ under constrained curvature flows}

\author{David Hartley}

\address{Instituto de Ciencias Matem\'aticas, Consejo Superior de
  Investigaciones Cient\'\i ficas, 28049 Madrid, Spain}

\email{david.hartley@icmat.es}

\date{}

\begin{document}

\begin{abstract}
In this paper we discuss the stability of geodesic spheres in $\SS^{n+1}$ under constrained curvature flows. We prove that under some standard assumptions on the speed and weight functions, the spheres are stable under perturbations that preserve a volume type quantity. This extends results of \cite{Escher98A} and \cite{Hartley15} to a Riemannian manifold setting.
\end{abstract}

\maketitle


\section{Introduction}\label{SecIntro}
We consider a family of hypersurfaces that are compact without boundary, $\{\Om_t\}_{t\in[0,T)}$, inside a Riemannian manifold $(N^{n+1},\gb)$ moving with a speed function $\hat{G}$ in the direction of the normal. If $\Om_0$ is given by an embedding $\tilde{X}_0:M^n\rightarrow N^{n+1}$ then the family is obtained by solving for an $X:M^n\times[0,T)\rightarrow N^{n+1}$ that satisfies
\begin{equation}\label{GeneralEq}
	\pat{X} = \hat{G}(X)\nu,\ X(\cdot,0)=\tilde{X}_0
\end{equation}
where $\nu$ is the outer unit normal to $\Om_t$, with $\Om_t=X(M^n,t)$. We will consider speed functions, $\hat{G}$, of the form:
\begin{equation}\label{GhatForm}
\hat{G}(X):=\frac{\int_{M^n}F(\ka)\hat{\Xi}(X)\dmu{}}{\int_{M^n}\hat{\Xi}(X)\dmu{}} -F(\ka),
\end{equation}
where $\dmu{}$ and $\ka=(\ka_1,\ldots,\ka_n)$ are, respectively, the induced volume form and principal curvatures (i.e. eigenvalues of the Weingarten map $\Wcal$) on $X(\Om)$, $\hat{\Xi}$ is a weight function, and $F$ is a smooth, symmetric function on $\RR^n$ satisfying $\frac{\pa F}{\pa \ka_a}(\ka(\tilde{X}_0))>0$.

We will consider a fairly general form of the weight function $\hat{\Xi}$, however of special interest are the cases when
\begin{equation}\label{XiFormEq}
\hat{\Xi}(X)=\sum_{a=0}^{n+1}c_a\hat{\Xi}_a(X)
\end{equation}
for some $\{c_0,\ldots,c_{n+1}\}\in\RR^{n+2}$, where
\begin{equation}\label{XiaFormEq}
\hat{\Xi}_a(X):=\left\{\begin{array}{ll} -\frac{g^{ik}}{(n+1)\binom{n}{n-a}}\left(\frac{\pa E_{n-a}}{\pa h^i_j}\gb(\bar{R}(\nu,T_k)\nu,T_j) +\nabla_j\nabla_k\left(\frac{\pa E_{n-a}}{\pa h^i_j}\right)\right) +\binom{n+1}{n+1-a}^{-1}E_{n+1-a} & \text{if }a=0,\ldots,n,\\ 1 & \text{if }a=n+1,\end{array}\right.
\end{equation}
$g$ is the induced metric on $X(M^n)$, $\bar{R}(U,W)Z=\nablab_W\nablab_UZ-\nablab_U\nablab_WZ+\nablab_{[U,W]}Z$ is the Riemann curvature tensor of $(N^{n+1},\gb)$, $\nablab$ and $\nabla$ are the Levi-Civita connections on $(N^{n+1},\gb)$ and induced on $(X(M^n),g)$ respectively, and
\[
E_a:=\sum_{1\leq b_1<\ldots<b_a\leq n}\prod_{i=1}^a\ka_{b_i},
\]
are the elementary symmetric functions of the Weingarten map. Note that for hypersurfaces in Euclidean space $g^{ik}\nabla_k\left(\frac{\pa E_{n-a}}{\pa h^i_j}\right)=0$; see the proof of Lemma 2.2.2. in \cite{HartleyPhD}. With $\hat{\Xi}$ as in (\ref{XiFormEq}), the flow (\ref{GeneralEq}) preserves the real valued quantity
\begin{equation}\label{VhatForm}
\hat{V}(\Om):=\sum_{a=0}^{n+1}c_a\hat{V}_a(\Om),
\end{equation}
where $\hat{V}_a$ are the mixed volumes
\[
\hat{V}_a(\Om):=\left\{\begin{array}{ll} \frac{1}{(n+1)\binom{n}{n-a}}\int_{M^n}E_{n-a}\dmu{} & \text{if }a=0,\ldots,n,\\ \text{Vol}(\Om) & \text{if }a=n+1.\end{array}\right.
\]
The topic of intrinsic volumes is more complicated in spherical space than Euclidean space. For example, in \cite{Gao03} they consider three different definitions, however each is a linear combination of the mixed volumes defined above and hence can be preserved by choosing the $c_a$ constants appropriately. See Appendix \ref{SecWeight} for a proof that $\hat{V}$ is preserved under the flow when $\hat{\Xi}$ is given by (\ref{XiFormEq}).

This flow in Euclidean space, and with a weight function such that a mixed volume is preserved, has been studied previously by McCoy in \cite{McCoy05}. There it was proved that under some additional conditions on $F$, for example homogeneity of degree one and convexity or concavity, initially convex hypersurfaces admit a solution for all time and that the hypersurfaces converge to a sphere as $t\to\infty$. This was an extension of a result by Huisken \cite{Huisken87} who proved the result for the volume preserving mean curvature flow (VPMCF). The stability of spheres has previously been considered by Escher and Simonett in \cite{Escher98A} for the case of the VPMCF in Euclidean space where it was proved that they were stable under small perturbations in the little H\"older space $h^{1,\alpha}$, any $\alpha\in(0,1)$. This result was extended by the author, \cite{Hartley15B}, to the case of mixed-volume preserving curvature flows, with the perturbations in the space $h^{2,\alpha}$ to account for the fully nonlinear nature of the flows. 

For flows of this nature in Riemannian manifolds Huisken noted in \cite{Huisken87} that even the VPMCF in $\mathbb{S}^{n+1}$ will, in general, not preserve the convexity of a hypersurface, thus making the standard analysis more difficult. In \cite{Alikakos03} Alikakos and Freire prove that if the manifold $N^{n+1}$ has a finite number of critical points of the scalar curvature that are all non-degenerate, then the geodesic spheres close to the critical points are stable under volume preserving perturbations. In the case when $N^{n+1}$ is hyperbolic space, Cabezas-Rivas and Vicente in \cite{Cabezas07} prove that the VPMCF of hypersurfaces that satisfy a certain convexity property exist for all time and converge to geodesic spheres. They also prove that geodesic spheres in these manifolds are stable with respect to the VPMCF under $h^{1,\alpha}$ perturbations.

In this paper we consider the stability of geodesic spheres in $\mathbb{S}^{n+1}$ under the flow (\ref{GeneralEq}). The main theorem is:
\begin{theorem}\label{MainTheorem}
A geodesic sphere $\Scal\subset\SS^{n+1}$ is stable under perturbations in $h^{2,\al}$, for any $\al\in(0,1)$, with respect to the flow (\ref{GeneralEq}), with $\hat{G}$ as in (\ref{GhatForm}), if the following hold:
\begin{itemize}
	\item $F$ is a smooth, symmetric function of the principal curvatures,
	\item $\frac{\pa F}{\pa\ka_1}(\ka(\Scal))>0$, and
	\item $\hat{\Xi}(\Scal)=const\neq0$.
\end{itemize}
To be precise let $\Om_0$ be a $h^{2,\al}$-close normal geodesic graph over $\Scal$, then the flow by (\ref{GeneralEq}) exists for all time and the hypersurfaces $\Om_t:=X(M^n,t)$ converge in $h^{2,\al}$ to a geodesic sphere close to $\Scal$.
\end{theorem}

Note that the little H\"older spaces on a manifold, $h^{k,\alpha}(M^n)$ for $k\in\mathbb{N}_0$ and $\al\in(0,1)$, are defined as the completion of the smooth functions inside the standard H\"older space $C^{k,\al}(M^n)$. They are useful in analysing stability properties as they obey a self interpolation property, \cite[Equation 19]{Guenther02}.

The paper is structured as follows. In Section \ref{SecNormal} we consider the properties of hypersurfaces that are normal geodesic graphs over a base hypersurface. We also give an equation on the space of graph functions that is equivalent to (\ref{GeneralEq}). Section \ref{SecLin} derives the linearisation of the speed function in the case where the hypersurfaces are graphs over a geodesic sphere in $\SS^{n+1}$. The space of functions that define geodesic spheres close to the base sphere is then analysed in Section \ref{SecSpheres}, and finally Theorem \ref{MainTheorem} is proved in Section \ref{SecProof}. In Appendix \ref{SecWeight} we derive the formula for the weight function so that the quantity $\hat{V}$ is preserved under the flow.


\section{Normal Geodesic Graphs and Equivalent Equations}\label{SecNormal}
We will consider the situation where the hypersurfaces are normal (geodesic) graphs over a base hypersurface $X_0(M^n)$ with outer unit normal $\nu_0$. These hypersurfaces can be written as $X_u(p)=\ga(p,u(p))$ where $u:M^n\rightarrow\RR$, and $\ga_p(s):=\ga(p,s)$ is the unique unit speed geodesic satisfying $\ga_p(0)=X_0(p)$ and $\gad_p(0)=\nu_0(p)$, we use a dot to denote a derivative with respect to the geodesic parameter, $s$. Note that we require $\|u\|_{C^0}<C_0$, where $C_0$ is the injectivity radius of $X_0(M^n)\subset N^{n+1}$.

Standard calculations give us the following formulas
\begin{lemma}\label{GraphFormulas}
The tangent vectors for a normal graph $X_u(M^n)$ are given by
\[
T_i(u)=\left.\frac{\pa\ga}{\pa p^i}+\nabla_iu \gad\right|_{s=u},
\]
the induced metric components are given by
\[
g_{ij}(u)=\left.\gb\left(\frac{\pa\ga}{\pa p^i},\frac{\pa\ga}{\pa p^j}\right)\right|_{s=u}+\nabla_iu\nabla_ju,
\]
and the unit normal by
\[
\nu(u)=\left.\sqrt{1-|\nabla u|_{g(u)}^2}\gad -\frac{g^{ij}(u)\nabla_iu}{\sqrt{1-|\nabla u|_{g(u)}^2}}\frac{\pa\ga}{\pa p^j}\right|_{s=u},
\]
where $g^{ij}(u)$ are the components of the inverse of $g(u)$, also note that ${|\nabla u|_{g(u)}=\sqrt{g^{ij}(u)\nabla_iu\nabla_ju}<1}$.
\end{lemma}

\begin{proof}
The formula for the tangent vectors follows directly from $X_u(p)=\ga(p,u(p))$. The metric formula then follows from using the unit speed condition $\gb(\gad,\gad)=1$ and the formula $\gb\left(\frac{\pa\ga}{\pa p^i},\gad\right)=0$. In fact,
\[
\frac{\pa}{\pa s}\left(\gb\left(\frac{\pa\ga}{\pa p^i},\gad\right)\right)=\gb\left(\nablab_{\frac{\pa\ga}{\pa p^i}}\gad,\gad\right)+\gb\left(\frac{\pa\ga}{\pa p^i},\nablab_{\gad}\gad\right)=0,
\]
where $\nablab_VU=V(U)+\Gab^{\al}_{\be\ga}V^{\be}U^{\ga}\pa_{\al}$ is the Levi-Civita connection on $(N^{n+1},\gb)$, and we have used the space derivative of the unit speed condition and the geodesic condition. Hence $\gb\left(\frac{\pa\ga}{\pa p^i},\gad\right)=\left.\gb\left(\frac{\pa\ga}{\pa p^i},\gad\right)\right|_{s=0}=\gb\left(\frac{\pa X_0}{\pa p^i},\nu_0\right)=0$, so
\[
g_{ij}(u)=\gb\left(T_i(u),T_j(u)\right)=\left.\gb\left(\frac{\pa\ga}{\pa p^i},\frac{\pa\ga}{\pa p^j}\right)+\nabla_iu\nabla_ju\gb(\gad,\gad)\right|_{s=u}.
\]

The formula for the unit normal can be seen by taking its inner product with the tangent vectors and using that $\left.\gb\left(\frac{\pa\ga}{\pa p^i},\frac{\pa\ga}{\pa p^j}\right)\right|_{s=u}=g_{ij}(u)-\nabla_iu\nabla_ju$.
\end{proof}

We now aim to show that the equation (\ref{GeneralEq}) is equivalent to an equation on $C^2(M^n)$. We define $L(u):=\gb(\gad|_{s=u},\nu(u))^{-1}=\frac{1}{\sqrt{1-|\nabla u|_{g(u)}^2}}$ and $G(u):=L(u)\hat{G}(X_u)$, and consider the flow
\begin{equation}\label{TanGenEq}
\frac{\pa u}{\pa t}=G(u),\ u(0)=u_0,
\end{equation}
then we have
\[
\frac{\pa X_u}{\pa t}=\frac{\pa u}{\pa t}\gad|_{s=u}=G(u)\gad|_{s=u},
\]
and in particular
\[
\left\langle\frac{\pa X_u}{\pa t},\nu(u)\right\rangle=\hat{G}(X_u).
\]
Therefore there is a tangential diffeomorphism $\phi_t:M^n\rightarrow M^n$, with $\phi_0=id$, such that $X_u(\phi_t(p),t)$ satisfies (\ref{GeneralEq}), with initial embedding $X_{u_0}$. Likewise if $X$ satisfies (\ref{GeneralEq}) and $X=X_u$ for some $u:M^n\times[0,T)\rightarrow\RR$ then from $\frac{\pa X_u}{\pa t}=\frac{\pa u}{\pa t}\gad|_{s=u}$ we obtain, via the inner product with $\nu(u)$, that $\hat{G}(X_u)=L(u)^{-1}\frac{\pa u}{\pa t}$ and hence $u$ satisfies (\ref{TanGenEq}). Therefore equations (\ref{TanGenEq}) and (\ref{GeneralEq}) are equivalent.

Note that because (\ref{TanGenEq}) is equivalent to (\ref{GeneralEq}), up to a tangential diffeomorphism, when $\hat{\Xi}$ is given by (\ref{XiFormEq}) we have that $V(u):=\hat{V}(X_u(M^n))$ is also a preserved quantity for (\ref{TanGenEq}). This can also be seen by using equation (\ref{EqVhatLin}) in Appendix \ref{SecWeight} to calculate the linearisation of $V(u)$:
\begin{align}\label{EqVLin}
DV(u)[w]=&D\hat{V}(X_u)[DX_u[w]]\nonumber\\
=&D\hat{V}(X_u)[\gad|_{s=u}w]\nonumber\\
=&\int_{M^n}\hat{\Xi}(X_u)\gb(\gad|_{s=u}w,\nu(u))\dmu{u}\nonumber\\
=&\int_{M^n}\hat{\Xi}(X_u)L(u)^{-1}w\dmu{u}.
\end{align}
Now by setting $w=\frac{\pa u}{\pa t}=L(u)\hat{G}(X_u)$ and using the form of $\hat{G}$ in (\ref{GhatForm}) we obtain $\frac{\pa V}{\pa t}=DV(u)\left[\frac{\pa u}{\pa t}\right]=0$.


\section{Linearisation about Geodesic Spheres in $\SS^{n+1}$}\label{SecLin}
We start by giving some standard linearisation formulas, and we suppress that quantities are to be evaluated at $X$,
\begin{lemma}\label{LemStandLin}
The components of the metric of $X(M^n)$, $g_{ij}(X):=\gb\left(\frac{\pa X}{\pa p^i},\frac{\pa X}{\pa p^j}\right)$, have the linearisation
\[
Dg_{ij}(X)[Y]=\gb\left(\nablab_{T_i}Y,T_j\right) +\gb\left(\nablab_{T_j}Y,T_i\right),
\]
where $T_i(X):=\frac{\pa X}{\pa p^i}$, and $\nablab_{T_i}Y=\frac{\pa Y}{\pa p^i}+\Gab_{\be\ga}^{\al}Y^{\be}T_i^{\ga}\pa_{\al}$. The volume element $\mu(X):=\sqrt{\det\left(g(X)\right)}$ has the linearisation
\[
D\mu(X)[Y]=g^{ij}\gb\left(\nablab_{T_i}Y,T_j\right)\mu,
\]
the unit normal of $X(M^n)$ satisfies
\[
D\nu(X)[Y]+\Gab^{\al}_{\be\ga}\nu^{\be}Y^{\ga}\pa_{\al}=-g^{ij}\gb\left(\nablab_{T_i}Y,\nu\right)T_j,
\]
the linearisation of the components of the second fundamental form, $h_{ij}(X):=\gb\left(\nablab_{T_i(X)}\nu(X),T_j(X)\right)$, is
\[
Dh_{ij}(X)[Y]=-\gb\left(\nablab_{T_i}\nablab_{T_j}Y,\nu\right) -\gb\left(\bar{R}\left(Y,T_i\right)\nu,T_j\right),
\]
and finally the linearisation of the elements of the Weingarten map, $h^i_j(X)=g^{ik}(X)h_{kj}(X)$, are given by
\[
Dh^i_j(X)[Y]=-g^{ik}\left(h^{l}_j\left(\gb\left(\nablab_{T_k}Y,T_l\right)+\gb\left(\nablab_{T_l}Y,T_k\right)\right) +\gb\left(\nablab_{T_k}\nablab_{T_j}Y,\nu\right) +\gb\left(\bar{R}\left(Y,T_k\right)\nu,T_j\right)\right).
\]
\end{lemma}

\begin{proof}
The variations of the metric and volume element are straight from the definitions. By taking the linearisations of the formulas
\[
\gb\left(\nu,T_i\right)=0,\ \text{and}\ \gb\left(\nu,\nu\right)=1,
\]
to obtain
\[
\gb\left(D\nu(X)[Y]+\Gab^{\al}_{\be\ga}\nu^{\be}Y^{\ga}\frac{\pa}{\pa q^{\al}},T_i\right)+\gb\left(\nu,\nablab_{T_i}Y\right)=0,\ \text{and}\ 2\gb\left(D\nu(X)[Y]+\Gab^{\al}_{\be\ga}\nu^{\be}Y^{\ga}\frac{\pa}{\pa q^{\al}},\nu\right)=0,
\]
we are able to conclude the formula for the linearisation of the unit normal.

Now we consider the linearisation of the second fundamental form:
\begin{align*}
Dh_{ij}(X)[Y]=&\gb\left(D\left(\nablab_{T_i}\nu\right)(X)[Y]+\Gab^{\al}_{\be\ga}\left(\nablab_{T_i}\nu\right)^{\be}Y^{\ga}\frac{\pa}{\pa q^{\al}},T_j\right) +\gb(\nablab_{T_i}\nu,\nablab_{T_j}Y)\\
=&\gb\left(\nablab_{T_i}\left(D\nu(X)[Y]+\Gab^{\al}_{\be\ga}\nu^{\be}Y^{\ga}\frac{\pa}{\pa q^{\al}}\right) +\bar{R}\left(T_i,Y\right)\nu,T_j\right) +\gb\left(\nablab_{T_i}\nu,\nablab_{T_j}Y\right)\\
=&\gb\left(\nablab_{T_i}\left(-g^{kl}\gb\left(\nablab_{T_k}Y,\nu\right)T_l\right),T_j\right) -\gb\left(\bar{R}\left(Y,T_i\right)\nu,T_j\right) +\gb\left(\nablab_{T_i}\nu,\nablab_{T_j}Y\right)\\
=&-\gb\left(\nablab_{T_i}\nablab_{T_j}Y,\nu\right) -\gb\left(\bar{R}\left(Y,T_i\right)\nu,T_j\right).
\end{align*}

Finally we use that $h^i_j=g^{ik}h_{kj}$ and that $Dg^{ik}(X)[Y]=-g^{ip}g^{kq}Dg_{pq}(X)[Y]$ to obtain the linearisation of the Weingarten map components
\begin{align*}
D h^i_j(X)[Y]=&-g^{ip}g^{kq}\left(\gb\left(\nablab_{T_p}Y,T_q\right) +\gb\left(\nablab_{T_q}Y,T_p\right)\right)h_{kj} -g^{ik}\left(\gb\left(\nablab_{T_k}\nablab_{T_j}Y,\nu\right) +\gb\left(\bar{R}\left(Y,T_k\right)\nu,T_j\right)\right)\\
=&-g^{ik}\left(h^{l}_j\left(\gb\left(\nablab_{T_k}Y,T_l\right)+\gb\left(\nablab_{T_l}Y,T_k\right)\right) +\gb\left(\nablab_{T_k}\nablab_{T_j}Y,\nu\right) +\gb\left(\bar{R}\left(Y,T_k\right)\nu,T_j\right)\right).
\end{align*}
\end{proof}

\begin{lemma}\label{GenLin}
\[
D\hat{G}(X)[Y]=\frac{\int_{M^n}\sum_{a=1}^n\frac{\pa F}{\pa \ka_a}(\ka(X))D\ka_a(X)[Y]\hat{\Xi}(X)-\frac{\hat{G}(X)}{\hat{\mu}(X)}D\left(\hat{\Xi}\hat{\mu}\right)(X)[Y]\,d\mu}{\int_{M^n}\hat{\Xi}(X)\,d\mu} -\sum_{a=1}^n\frac{\pa F}{\pa \ka_a}(\ka(X))D\ka_a(X)[Y],
\]
\[
D\ka_a(X_0)[w\nu_0]=-\zez_a^i\zez_a^j\left(\nablaz_i\nablaz_jw+\gb\left(\bar{R}(\nu_0,\Tz_i)\nu_0,\Tz_j\right)w\right)-\mathring{\ka}_a^2w,
\]
where $\Tz_i$ are tangent vectors, $\nablaz$ is the Levi-Civita connection, and $\zez_a$ is the principle direction (eigenvector of the Weingarten map $\mathring{\Wcal}$) corresponding to the principle curvature $\mathring{\ka}_a$ of $X_0(M^n)$. In particular if $X_0(M^n)$ is totally umbilic we have
\[
\sum_{a=1}^n\frac{\pa F}{\pa \ka_a}(\mathring{\ka})D\ka_a(X_0)[w\nu_0]=-\frac{\pa F}{\pa \ka_1}(\mathring{\ka})\left(\De_{\gz}w+|\mathring{\Wcal}|_{\gz}^2w+\bar{Ric}(\nu_0,\nu_0)w\right),
\]
where $\gz$ is the metric of $X_0(M^n)$.
\end{lemma}

\begin{proof}
The first formula follows directly from the definition of $\hat{G}(X)$, while the second formula follows from $D\ka_a(X)[Y]=\ze^i(X)\ze^j(X)g_{jk}Dh_{i}^k(X)[Y]$ and Lemma \ref{LemStandLin}.
\end{proof}

Now we consider $N^{n+1}=\SS^{n+1}$ with coordinates $q_{\al}$, $\al=1,\ldots,n+1$, with $q_{1}\in[0,2\pi)$ and $q_{\al}\in[0,\pi)$ for $\al=2,\ldots,n+1$, such that the metric is
\[
\gb=g_{\SS^{n+1}}=\sum_{\al=1}^{n+1}\prod_{\be=\al+1}^{n+1}\sin(q_{\be})^2\,dq^{\al}{}^2,
\]
and $X_0(M^n)$ equal to the $n$-sphere $q_{n+1}=\theta$, for some fixed $\theta\in(0,\pi)$, which we denote $\Scal_{\theta}$. With this set up normal graphs take the form
\begin{equation}\label{SphereGraph}
X_u(p)=\left(p_1,\ldots,p_n,\theta+u(p)\right),
\end{equation}
where $p=(p_1,\ldots,p_n)\in\SS^n$ and $C_0=\min(\te,\pi-\te)$. The tangent vectors, metric, and Weingarten map for $X_0(M^n)$ are then
\[
\Tz_i=\de^{\al}_i\frac{\pa}{\pa q^{\al}},
\]
\[
\gz_{ij}=\sin(\theta)^2g_{\SS^n}=\sin(\theta)^2\sum_{i=1}^{n}\prod_{j=i+1}^{n}\sin(p_j)^2\,dp^i{}^2,
\]
\[
\mathring{\Wcal}=\cot(\theta)Id.
\]

\begin{lemma}\label{Weight0}
\[
\Xi_a(X_0)=\left\{\begin{array}{ll} \frac{\cot(\theta)^{n-a-1}}{n+1}\left(a\cot(\te)^2-n+a\right) & \text{if }a=0,\ldots,n,\\ 1 & \text{if }a=n+1.\end{array}\right.
\]
\end{lemma}

\begin{proof}
This follows from a straightforward calculation using (\ref{XiFormEq}). Firstly we note that since $\nablaz_k\hz^i_j=0$ we have
\[
\nablaz_j\nablaz_k\left(\frac{\pa E_{n-a}}{\pa h^i_j}(\mathring{\ka})\right)=\nablaz_j\left(\frac{\pa^2 E_{n-a}}{\pa h^i_j\pa h^p_q}(\mathring{\ka})\nablaz_k\hz^p_q\right)=0.
\]
Next we use that $\gb\left(\bar{R}(\nu_0,\Tz_k)\nu_0,\Tz_j\right)=\gz_{kj}$, $E_a(\mathring{\ka})=\binom{n}{a}\cot(\theta)^a$, and the formula (see Proposition B.0.2. in \cite{HartleyPhD} for example)
\begin{equation}\label{EDeriv}
\frac{\pa E_a}{\pa h^i_j}(\ka)=\sum_{b=0}^{a-1}(-1)^b\left(\Wcal^b\right)_i^jE_{a-1-b}(\ka),
\end{equation}
to calculate the remaining terms:
\begin{align*}
\Xi_a(X_0)=&\frac{-\gz^{ik}}{(n+1)\binom{n}{n-a}}\frac{\pa E_{n-a}}{\pa h^i_j}(\mathring{\ka})\gz_{kj} +\frac{E_{n-a+1}(\mathring{\ka})}{\binom{n+1}{n-a+1}}\\
=&\frac{-\de^i_j}{(n+1)\binom{n}{n-a}}\sum_{b=0}^{n-a-1}(-1)^b\left(\mathring{\Wcal}^b\right)_i^jE_{n-a-1-b}(\mathring{\ka}) +\frac{\binom{n}{n+1-a}\cot(\theta)^{n+1-a}}{\binom{n+1}{n+1-a}}\\
=&\frac{-1}{(n+1)\binom{n}{n-a}}\sum_{b=0}^{n-a-1}(-1)^b\text{tr}\left(\mathring{\Wcal}^b\right)\binom{n}{n-a-1-b}\cot(\theta)^{n-a-1-b} +\frac{a}{n+1}\cot(\theta)^{n+1-a}\\
=&\frac{-n}{(n+1)\binom{n}{n-a}}\sum_{b=0}^{n-a-1}(-1)^b\binom{n}{n-a-1-b}\cot(\theta)^{n-a-1} +\frac{a}{n+1}\cot(\theta)^{n+1-a}\\
=&\frac{a-n}{n+1}\cot(\theta)^{n-a-1} +\frac{a}{n+1}\cot(\theta)^{n+1-a}.
\end{align*}
\end{proof}

\begin{remark}
\begin{itemize}
	\item For our purposes the important thing here is that each $\Xi_a(X_0)$ is a constant and, hence, the form (\ref{XiFormEq}) will be allowable in our theorem, provided $c_a\in\RR$, $a=0,\ldots,n+1$, are such that $\hat{\Xi}(X_0)\neq0$.
	\item It should be noted that we have $\Xi_a(X_0)=0$ if and only if $\theta\in\{\theta_a,\frac{\pi}2,\pi-\theta_a\}$, where $\theta_a=\arcsin\left(\sqrt{\frac{a}{n}}\right)$, except in the cases of $a=n-1$ when it is if and only if $\theta\in\{\theta_a,\pi-\theta_a\}$ and $a=n+1$ when it is never zero.
	\item As suggested in \cite{Gao03} a better intrinsic volume to use in spherical space may be
	\[ \hat{U}_a(\Om)=\frac{\Ga(n+2)}{2^{n+2}\pi^{\frac{n}{2}}}\sum_{b=0}^{\left\lfloor\frac{n+1-a}{2}\right\rfloor}\frac{\hat{V}_{a+2b}(\Om)}{\Ga\left(\frac{a+2b}{2}+1\right)\Ga\left(\frac{n-a-2b}{2}+1\right)},\]
	where $\Ga$ is the usual Gamma function, for which the linearisation is
	\[ D\hat{U}_a(X)[Y]=\int_{M^n}\hat{Z}_a(X)\gb(Y,\nu)\dmu,\]
	where
	\[\hat{Z}_a(X)=\frac{\Ga(n+2)}{2^{n+2}\pi^{\frac{n}{2}}}\sum_{b=0}^{\left\lfloor\frac{n+1-a}{2}\right\rfloor}\frac{\hat{\Xi}_{a+2b}(X)}{\Ga\left(\frac{a+2b}{2}+1\right)\Ga\left(\frac{n-a-2b}{2}+1\right)}.\]
	In this case the linearisation functions at $X_0$ have the simpler form
	\[ \hat{Z}_a(X_0)=\frac{a\Ga(n+1)\cos(\te)^{n+1-a}}{2^{n+2}\pi^{\frac{n}{2}}\Ga\left(\frac{a}{2}+1\right)\Ga\left(\frac{n-a}{2}+1\right)},\]
	and is zero if and only if $\te=\frac{\pi}{2}$ and $a\neq n+1$.
\end{itemize}
\end{remark}

\begin{lemma}
We assume that $\hat{\Xi}(X_0)=const\neq0$, then
\[
DG(0)[w]=\frac{\pa F}{\pa \ka_1}(\mathring{\ka})\sin(\theta)^{-2}\left(\De_{g_{\SS^n}}w+nw-\fint_{\SS^n}w\dmu{0}\right)
\]
\end{lemma}

\begin{proof}
To calculate $DG(0)$ we first note that $DL(0)[w]=0$ and $DX_u|_{u=0}[w]=w\nu_0$, so that
\begin{align*}
DG(0)[w]=&D\hat{G}(X_0)[w\nu_0]\\
=&\frac{\int_{\SS^n}\sum_{a=1}^n\frac{\pa F}{\pa \ka_a}(\mathring{\ka})D\ka_a(X_0)[w\nu_0]\hat{\Xi}(X_0)-\frac{\hat{G}(X_0)}{\hat{\mu}(X_0)}D\left(\hat{\Xi}\hat{\mu}\right)(X_0)[w\nu_0]\dmu{0}}{\int_{\SS^n}\hat{\Xi}(X_0)\dmu{0}}-\sum_{a=1}^n\frac{\pa F}{\pa \ka_a}(\mathring{\ka})D\ka_a(X_0)[w\nu_0]\\
=&\frac{\pa F}{\pa \ka_1}(\mathring{\ka})\left(\De_{\gz}w+|\mathring{\Wcal}|_{\gz}^2w+\bar{Ric}(\nu_0,\nu_0)w\right) -\fint_{\SS^n}\frac{\pa F}{\pa \ka_1}(\mathring{\ka})\left(\De_{\gz}w+|\mathring{\Wcal}|_{\gz}^2w+\bar{Ric}(\nu_0,\nu_0)w\right)\dmu{0}\\
=&\frac{\pa F}{\pa \ka_1}(\mathring{\ka})\left(\De_{\gz}w+|\mathring{\Wcal}|_{\gz}^2w+\bar{Ric}(\nu_0,\nu_0)w -\fint_{\SS^n}\De_{\gz}w+|\mathring{\Wcal}|_{\gz}^2w+\bar{Ric}(\nu_0,\nu_0)w\dmu{0}\right)\\
=&\frac{\pa F}{\pa \ka_1}(\mathring{\ka})\left(\sin(\theta)^{-2}\De_{g_{\SS^n}}w+n\cot(\theta)^2w+nw -\fint_{\SS^n}n\cot(\theta)^2w+nw\dmu{0}\right)\\
=&\frac{\pa F}{\pa \ka_1}(\mathring{\ka})\sin(\theta)^{-2}\left(\De_{g_{\SS^n}}w+nw-n\fint_{\SS^n}w\dmu{0}\right)
\end{align*}
\end{proof}

\begin{corollary}\label{LinProp}
Let $c_a\in\RR$, $a=0,\ldots,n+1$, be such that $\hat{\Xi}(X_0)=const\neq0$, and the smooth, symmetric function $F$ be such that $\frac{\pa F}{\pa \ka_1}(\mathring{\ka})>0$, then
\[
\sup\left\{\la:\la\in\si(DG(0))\backslash\{0\}\right\}<0,
\]
and $0$ is an eigenvalue of $DG(0)$ with multiplicity $n+2$ and eigenfunctions given by the first order spherical harmonics on $\SS^n$, $Y_1^{(n)},\ldots,Y_{n+1}^{(n)}$, and the constant function.
\end{corollary}

\begin{remark}
Our coordinates give the following formula for the spherical harmonics
\[
Y_i^{(n)}(p)=\prod_{j=i}^{n}\sin(p_j)\cos(p_{i-1}).
\]
\end{remark}


\section{Space of Geodesic Spheres in $\SS^{n+1}$}\label{SecSpheres}
In this section we consider how to parameterise the space of geodesic spheres as normal geodesic graphs over a particular sphere. That is we look to find the space of functions such that $X_u$ is a sphere near $X_0$ for any $u$ in the space, and that all spheres near $X_0$ are accounted for. To find the parametrisation we embed $\SS^{n+1}$ in $\RR^{n+2}$ using the first order spherical harmonics on $\SS^{n+1}$:
\[
Z(q)=\left(Y_{1}^{(n+1)}(q),\ldots,Y_{n+2}^{(n+1)}(q)\right).
\]
The image of $X_0$ lies in the plane $x_{n+2}=\cos(\theta)$, so any non-perpendicular sphere is determined by the values $b_i$, $i=1,\ldots,n+2$, such that it lies in the plane
\[
x_{n+2}+\sum_{i=1}^{n+1}b_ix_i=\sqrt{1+|b|^2}(\cos(\te)+b_{n+2}),
\]
for $(b_1,\ldots,b_{n+2})\in\RR^{n+1}\times(-1-\cos(\te),1-\cos(\te))$, where $|b|=\left(\sum_{i=1}^{n+1}b_i^2\right)^{\frac12}$. The radius of the sphere is given by $\tilde{R}=\tilde{R}(b_{n+2})=\sqrt{\sin(\te)^2-2\cos(\te)b_{n+2}-b_{n+2}^2}$. Using the form of the normal graph (\ref{SphereGraph}) we find that $u$ satisfies
\[
\cos(\theta+u(p))+\sin(\theta+u(p))\sum_{i=1}^{n+1}b_i\prod_{j=i}^{n}\sin(p_j)\cos(p_{i-1})=\sqrt{1+|b|^2}(\cos(\te)+b_{n+2}),
\]
so by using the formula for the spherical harmonics on $\SS^n$ and solving for $u$ we obtain that the graph functions for the spheres are given by

\begin{equation}\label{uSphereForm}
u_b=\arctan\left(\frac{\sum_{i=1}^{n+1}b_iY_i^{(n)}+\sqrt{1+|b|^2}(\cos(\te)+b_{n+2})\sqrt{1+\left(\sum_{i=1}^{n+1}b_iY_i^{(n)}\right)^2-(1+|b|^2)(\cos(\theta)+b_{n+2})^2}}{(1+|b|^2)(\cos(\theta)+b_{n+2})^2-\left(\sum_{i=1}^{n+1}b_iY_i^{(n)}\right)^2}\right)-\theta,
\end{equation}
where $\arctan$ is defined such that $\arctan:\RR\rightarrow[0,\pi)$ and we require $|b|^2<\frac{1}{(\cos(\te)+b_{n+2})^2}-1$. This requirement means that the geodesic spheres considered divide the poles $q_{n+1}=0$ and $q_{n+1}=\pi$.

Note $u_0=0$ so this gives the base geodesic sphere, and if we linearise, with respect to the parameters, at the base sphere we have
\[
Du_b|_{b=0}[z]=\sum_{j=1}^{n+1}z_jY_j^{(n)}+\frac{1}{\sin(\te)}z_{n+2}.
\]

We have thus proved the following
\begin{lemma}\label{SphereSpace}
The space of graph functions defining a sphere non-perpendicular to $X_0\left(M^n\right)$ and separating the poles is given by
\[
\mathscr{S}:=\{u_b\in C^0(\SS^n):b\in\RR^{n+1}\times(-1-\cos(\te),1-\cos(\te),\ |b|^2<\frac{1}{(\cos(\te)+b_{n+2})^2}-1\}
\]
where $u_b$ is defined as in (\ref{uSphereForm}). Further, at $b=0$ it has the tangent space $T_0\mathscr{S}=\text{span}\left(\{Y_i^{(n)}\in C^0(\SS^n):i=1,\ldots,n+1\}\cup\{1\}\right)$ and $\mathscr{S}$ is locally a differentiable graph over it.
\end{lemma}


\section{Proof of Main Theorem}\label{SecProof}
The proof of Theorem \ref{MainTheorem} follows precisely as in \cite{Hartley15} since $DG(0)$ is a positive multiple of the linear operator in that paper, see Lemma 3.1. of \cite{Hartley15} for its definition, and the stationary solutions are again a graph over $Null(DG(0))$. Firstly, since $DG(0)[w]+\frac{\pa F}{\pa \ka_1}(\mathring{\ka})\sin(\theta)^{-2}\fint_{\SS^n}w\dmu{0}$ is the negative of an elliptic operator, it is sectorial as a map from $h^{2,\al}(\SS^n)$ to $h^{0,\al}(\SS^n)$ for any $\al\in(0,1)$. Next, we use that $\frac{\pa F}{\pa \ka_1}(\mathring{\ka})\sin(\theta)^{-2}\fint_{\SS^n}w\dmu{0}$ is a bounded linear map from $h^{2,\al}(\SS^n)$ to $h^{2,\al}(\SS^n)$ to conclude that $DG(0):h^{2,\al}(\SS^n)\rightarrow h^{0,\al}(\SS^n)$ is sectorial for any $\al\in(0,1)$. Now, as being sectorial is a stable condition, this implies that $DG(w):h^{2,\al}(\SS^n)\rightarrow h^{0,\al}(\SS^n)$ is also sectorial for any $w$ in a neighbourhood $0\in O_{\al}\subset h^{2,\al}(\SS^n)$ and $\al\in(0,1)$. Short-time existence for (\ref{TanGenEq}) then follows directly from Theorem 8.4.1 in \cite{Lunardi95}.

\begin{theorem}\label{ShortTime}
For any $\al\in(0,1)$ there are constants $\delta,r>0$ such that if $\left\|u_0\right\|_{h^{2,\alpha}\left(\SS^n\right)}\leq r$ then equation (\ref{TanGenEq}) has a unique maximal solution:
\begin{equation*}
u\in C\left([0,\delta),h^{2,\alpha}\left(\SS^n\right)\right)\cap C^1\left([0,\delta),h^{0,\alpha}\left(\SS^n\right)\right).
\end{equation*}
\end{theorem}

Next we see the existence of a center manifold which attracts solutions, moreover, locally this is our space of stationary solution $\mathscr{S}$. Let $P$ be the spectral projection from $h^{2,\alpha}\left(\SS^n\right)$ onto $T_0\mathscr{S}$ associated with $DG(0)$, $\la_1$ be the first non-zero eigenvalue of $DG(0)$, and $\psi:U\subset T_0\mathscr{S}\rightarrow (I-P)\left[h^{2,\alpha}\left(\SS^n\right)\right]$ be the local graph function for $\mathscr{S}$.
\begin{lemma}\label{CenterManifold}
The space $\mathscr{S}$ is a local, invariant, exponentially attractive, center manifold for (\ref{TanGenEq}). In particular, there exists $r_1, r_2>0$ such that if $\|u_0\|_{h^{2,\alpha}\left(\SS^n\right)}<r_2$ then there exists $z_0\in T_0\mathscr{S}$ such that
\begin{equation}\label{ExpConv}
\|P[u(t)]-z(t)\|_{h^{0,\alpha}\left(\SS^n\right)}+\|(I-P)[u(t)]-\psi(z(t))\|_{h^{2,\alpha}\left(\SS^n\right)}\leq C\exp(-\om t)\|(I-P)[u_0]-\psi(P[u_0])\|_{h^{2,\alpha}\left(\SS^n\right)},
\end{equation}
for as long as $\|P[u(t)]\|_{h^{0,\alpha}\left(\SS^n\right)}<r_1$, where $\om\in(0,-\la_1)$, $C$ is a constant depending on $\om$, and
\begin{equation}\label{zEq}
z'(t)=P\left[G\left(\eta\left(\frac{z(t)}{r_1}\right)z(t)+\psi(z(t))\right)\right],\ z(0)=z_0,
\end{equation}
where $\eta:T_0\mathscr{S}\rightarrow\RR$ is a smooth cut off function such that $0\leq \eta(x)\leq 1$, $\eta(x)=1$ if $\|x\|_{h^{0,\alpha}\left(\SS^n\right)}\leq 1$ and $\eta(x)=0$ is $\|x\|_{h^{0,\alpha}\left(\SS^n\right)}\geq2$.
\end{lemma}

\begin{proof}
The existence of a local center manifold, $\mathcal{M}^c$, follows from Theorem 9.2.2 in \cite{Lunardi95}, where it is also shown to be a local graph over the nullspace of $DG(0)$, i.e. $T_0\mathscr{S}$. Theorem 2.3 in \cite{Iooss92} states that $\mathcal{M}^c$ contains all local stationary solutions, i.e. $\mathscr{S}\subset\mathcal{M}^c$, so combining these two facts we see that $\mathcal{M}^c=\mathscr{S}$. The exponential attractivity comes from Proposition 9.2.4 of \cite{Lunardi95}.
\end{proof}

Using (\ref{ExpConv}) evaluated at $t=0$ we obtain
\begin{align*}
\|z_0\|_{h^{0,\alpha}\left(\SS^n\right)}\leq& \|P[u_0]\|_{h^{0,\alpha}\left(\SS^n\right)} +\|P[u_0]-z_0\|_{h^{0,\alpha}\left(\SS^n\right)}\\
\leq &\|P[u_0]\|_{h^{0,\alpha}\left(\SS^n\right)} +C\|(I-P)[u_0]-\psi(P[u_0])\|_{h^{2,\alpha}\left(\SS^n\right)},
\end{align*}
so since $\psi$ is Lipschitz and $P$ is bounded, this leads to a bound of the form $\|z_0\|_{h^{0,\alpha}\left(\SS^n\right)}\leq C\|u_0\|_{h^{2,\alpha}\left(\SS^n\right)}$. Therefore we can ensure that $\|z_0\|_{h^{0,\alpha}\left(\SS^n\right)}<r_1$ by taking $\|u_0\|_{h^{2,\alpha}\left(\SS^n\right)}$ small enough, and since $z_0+\psi(z_0)$ defines a sphere, we see $G\left(\eta\left(\frac{z_0}{r_1}\right)z_0+\psi(z_0)\right)=G\left(z_0+\psi(z_0)\right)=0$. Hence $z(t)=z_0$ is the solution to (\ref{zEq}) and we can restate (\ref{ExpConv}) as
\begin{equation}\label{ConvEq}
\|P[u(t)]-z_0\|_{h^{0,\alpha}\left(\SS^n\right)}+\|(I-P)[u(t)]-\psi(z_0)\|_{h^{2,\alpha}\left(\SS^n\right)}\leq C\exp(-\om t)\|(I-P)[u_0]-\psi(P[u_0])\|_{h^{2,\alpha}\left(\SS^n\right)},
\end{equation}
for as long as $P[u(t)]\in B_{r_1}(0)$. However using this bound, and our bound for $z_0$, it follows that $\|P[u(t)]\|_{h^{0,\alpha}\left(\SS^n\right)}<C\|u_0\|_{h^{2,\alpha}\left(\SS^n\right)}$ as long as $\|P[u(t)]\|_{h^{0,\alpha}\left(\SS^n\right)}<r_1$. By choosing $\|u_0\|_{h^{2,\alpha}\left(\SS^n\right)}$ small enough we can therefore ensure $\|P[u(t)]\|_{h^{0,\alpha}\left(\SS^n\right)}<\frac{r_1}{2}$ for all $t\geq0$. Thus (\ref{ConvEq}) is true for all $t\geq0$ and this proves that $u(t)$ converges to $z_0+\psi(z_0)$ as $t\rightarrow\infty$. This completes the proof of Theorem \ref{MainTheorem} since $z_0+\psi(z_0)$ is the graph function of a sphere.

We also have the following corollary that follows by a simple continuity argument as in \cite[Corollary 3.8]{Guenther02} or \cite[Corollary 4.4]{Hartley15}.
\begin{corollary}
Let $\Om_0$ be a graph over a sphere with height function $u_0$ such that the solution, $u(t)$, to the flow (\ref{TanGenEq}) with initial condition $u_0$ exists for all time and converges to zero. Suppose further that $\left.\frac{\partial F}{\partial \kappa_i}\right|_{\kappa\left(X_{u(t)}\right)}>0$ for all $t\in[0,\infty)$ and $i=1,\ldots,n$. Then there exists a neighbourhood, $O$, of $u_0$ in $h^{2,\alpha}\left(\SS^n\right)$, $0<\alpha<1$, such that for every $w_0\in O$ the solution to (\ref{TanGenEq}) with initial condition $w_0$ exists for all time and converges to a function near zero whose graph is a sphere.
\end{corollary}


\appendix

\section{Form of the Weight Function}\label{SecWeight}
In this appendix we determine the form the weight function must take in order to preserve the quantity $\hat{V}$ in (\ref{VhatForm}). We start by considering the linearisation of the mixed volumes. We will abuse notation and set $\hat{V}(X)=\hat{V}(X(M^n))$ and $\hat{V}_a(X)=\hat{V}_a(X(M^n))$.
\begin{lemma}\label{VhatLin}
The mixed volumes have the linearisation
\[
D\hat{V}_a(X)[Y]= \int_{M^n} \hat{\Xi}_a(X)\gb\left(Y,\nu\right)\dmu{},
\]
for $a=0,\ldots,n+1$, with $\hat{\Xi}_a$ as defined in (\ref{XiaFormEq}).
\end{lemma}

\begin{proof}
We first note that the formula for $D\hat{V}_{n+1}(X)[Y]$ is standard.

Now we consider the linearisation of the mixed volumes with $0\leq a\leq n$, but before starting the calculation we state some useful relations for the elementary symmetric functions:
\begin{equation}\label{ELowHigh}
\frac{\pa E_a}{\pa h^i_j}=g_{ik}g^{jl}\frac{\pa E_a}{\pa h^l_k},
\end{equation}
\begin{equation}\label{ELowDegree}
\frac{\pa E_{a+1}}{\pa h^i_j}=E_a\de_i^j-h^j_k\frac{\pa E_a}{\pa h^i_k},
\end{equation}
and
\begin{equation}\label{ESwitch}
h_{ij}\frac{\pa E_a}{\pa h^k_j}=h_{kj}\frac{\pa E_a}{\pa h^i_j},
\end{equation}
which are all easily obtained from (\ref{EDeriv}).

We can now calculate the linearisation for $a=0,\ldots,n$ using Lemma \ref{LemStandLin}
\begin{align*}
(n+1)\binom{n}{a}D\hat{V}_{n-a}(X)[Y]=&\int_{M^n}\frac{\pa E_a}{\pa h_j^i}D h^i_j(X)[Y] +\frac{E_{a}}{\mu}D\mu(X)[Y]\dmu{}\\
=&\int_{M^n}-g^{ik}\frac{\pa E_{a}}{\pa h^i_j}\left(\gb\left(\nablab_{T_k}\nablab_{T_j}Y,\nu\right) +\gb\left(\bar{R}\left(Y,T_k\right)\nu,T_j\right)\right)\\
&\hspace{0.5cm} -g^{ik}\frac{\pa E_{a}}{\pa h^i_j}h^l_j\left(\gb\left(\nablab_{T_k}Y,T_l\right)+\gb\left(\nablab_{T_l}Y,T_k\right)\right) +E_{a}g^{ik}\gb\left(\nablab_{T_i}Y,T_k\right)\dmu{}\\
=&\int_{M^n} -g^{ik}\left(h_j^l\frac{\pa E_a}{\pa h^i_j}\gb\left(\nablab_{T_k}Y,T_l\right) +h_j^l\frac{\pa E_a}{\pa h^i_j}\gb\left(\nablab_{T_l}Y,T_k\right) -\nabla_k\left(\frac{\pa E_a}{\pa h^i_j}\right)\gb\left(\nablab_{T_j}Y,\nu\right)\right.\\
&\left.\hspace{1.4cm} -\frac{\pa E_a}{\pa h^i_j}\gb\left(\nablab_{T_j}Y,\nablab_{T_k}\nu\right) +\frac{\pa E_a}{\pa h^i_j}\gb\left(\bar{R}(Y,T_k)\nu,T_j\right) -E_a\gb\left(\nablab_{T_i}Y,T_k\right)\right)\dmu{}\\
=&\int_{M^n} \left(h^{iq}\frac{\pa E_a}{\pa h^i_p} -g^{ip}h_j^q\frac{\pa E_a}{\pa h^i_j} -g^{iq}h_j^p\frac{\pa E_a}{\pa h^i_j} +g^{pq}E_a\right)\gb\left(\nablab_{T_p}Y,T_q\right)  -g^{ik}\frac{\pa E_a}{\pa h^i_j}\gb\left(\bar{R}(Y,T_k)\nu,T_j\right)\\
&\hspace{0.5cm} -g^{ik}\nabla_j\nabla_k\left(\frac{\pa E_a}{\pa h^i_j}\right)\gb(Y,\nu) -g^{ik}\nabla_k\left(\frac{\pa E_a}{\pa h^i_j}\right)\gb\left(Y,\nablab_{T_j}\nu\right)\dmu{}.
\end{align*}
We now use Equation (\ref{ELowHigh}) to cancel the first two terms in the $\gb\left(\nablab_{T_p}Y,T_q\right)$ factor, and Equation (\ref{ESwitch}) to alter the third term of the factor:
\begin{align*}
(n+1)\binom{n}{a}D\hat{V}_{n-a}(X)[Y]=&\int_{M^n}\left(E_ag^{pq} -g^{iq}g^{pl}h_{ij}\frac{\pa E_{a}}{\pa h^l_j}\right)\gb\left(\nablab_{T_p}Y,T_q\right)  -g^{ik}\frac{\pa E_a}{\pa h^i_j}\gb\left(\bar{R}(Y,T_k)\nu,T_j\right)\\
&\hspace{0.5cm} -g^{ik}\nabla_j\nabla_k\left(\frac{\pa E_a}{\pa h^i_j}\right)\gb(Y,\nu) -g^{ik}h_j^l\nabla_k\left(\frac{\pa E_a}{\pa h^i_j}\right)\gb(Y,T_l)\dmu{}\\
=&\int_{M^n} -g^{ip}\left(\frac{\pa E_a}{\pa h^k_l}\nabla_p h^k_l\de_i^q -\nabla_p\left(\frac{\pa E_a}{\pa h^i_j}\right) h_j^q -\frac{\pa E_{a}}{\pa h^i_j}\nabla_p h_j^q\right)\gb(Y,T_q)\\
&\hspace{0.5cm}  -g^{ip}\left(E_a\de_i^q -h_j^q\frac{\pa E_{a}}{\pa h^i_j}\right)\gb\left(Y,\nablab_{T_p}T_q\right) -g^{ik}\nabla_j\nabla_k\left(\frac{\pa E_a}{\pa h^i_j}\right)\gb(Y,\nu)\\
&\hspace{0.5cm} -g^{ik}h_j^l\nabla_k\left(\frac{\pa E_a}{\pa h^i_j}\right)\gb(Y,T_l)   -g^{ik}\frac{\pa E_a}{\pa h^i_j}\gb\left(\bar{R}(Y,T_k)\nu,T_j\right)\dmu{}\\
=&\int_M -\left(g^{pq}g^{ik}\frac{\pa E_a}{\pa h^k_l}\nabla_p h_{il} -g^{ip}g^{ql}\frac{\pa E_{a}}{\pa h^i_j}\nabla_p h_{jl}\right)\gb(Y,T_q) +g^{ip}h_{pq}\frac{\pa E_{a+1}}{\pa h^i_q}\gb(Y,\nu)\\
&\hspace{0.5cm} -g^{ik}\nabla_j\nabla_k\left(\frac{\pa E_a}{\pa h^i_j}\right)\gb(Y,\nu) -g^{ik}\frac{\pa E_a}{\pa h^i_j}\gb\left(\bar{R}(Y,T_k)\nu,T_j\right)\dmu{}\\
=&\int_{M^n} -g^{ip}g^{ql}\frac{\pa E_{a}}{\pa h^i_j}\left(\nabla_l h_{pj} -\nabla_p h_{jl}\right)\gb(Y,T_q) +h^i_{q}\frac{\pa E_{a+1}}{\pa h^i_q}\gb(Y,\nu)\\
&\hspace{0.5cm} -g^{ik}\nabla_j\nabla_k\left(\frac{\pa E_a}{\pa h^i_j}\right)\gb(Y,\nu) -g^{ik}\frac{\pa E_a}{\pa h^i_j}\gb\left(\bar{R}(Y,T_k)\nu,T_j\right)\dmu{}.
\end{align*}
Now we use the homogeneity $E_a$ and the Gauss-Codazzi equation:
\begin{align*}
(n+1)\binom{n}{a}D\hat{V}_{n-a}(X)[Y]=&\int_{M^n} -g^{ip}g^{ql}\frac{\pa E_{a}}{\pa h^i_j}\gb\left(\bar{R}(T_l,T_p)T_j,\nu\right)\gb(Y,T_q) +(a+1)E_{a+1}\gb(Y,\nu)\\
&\hspace{0.5cm} -g^{ik}\nabla_j\nabla_k\left(\frac{\pa E_a}{\pa h^i_j}\right)\gb(Y,\nu) -g^{ik}\frac{\pa E_a}{\pa h^i_j}\gb\left(\bar{R}(Y,T_k)\nu,T_j\right)\dmu{}\\
=&\int_{M^n} -g^{ik}\frac{\pa E_{a}}{\pa h^i_j}\gb\left(\bar{R}\left(Y-g^{ql}\gb(Y,T_q)T_l,T_k\right)\nu,T_j\right) +(a+1)E_{a+1}\gb(Y,\nu)\\
&\hspace{0.5cm} -g^{ik}\nabla_j\nabla_k\left(\frac{\pa E_a}{\pa h^i_j}\right)\gb(Y,\nu)\dmu{}\\
=&\int_{M^n} -g^{ik}\frac{\pa E_{a}}{\pa h^i_j}\gb\left(\bar{R}\left(\gb(Y,\nu)\nu,T_k\right)\nu,T_j\right) +(a+1)E_{a+1}\gb(Y,\nu) -g^{ik}\nabla_j\nabla_k\left(\frac{\pa E_a}{\pa h^i_j}\right)\gb(Y,\nu)\dmu{}.
\end{align*}
\end{proof}

\begin{corollary}\label{XiForm}
If
\[
\hat{\Xi}(X)=\sum_{a=0}^{n+1}c_a\hat{\Xi}_a(X)
\]
for some constants $c_a\in\RR$, $a=0,\ldots,n+1$, where $\hat{\Xi}_a(X)$ are defined in (\ref{XiaFormEq}), then $\hat{V}(X)$ is preserved by the flow (\ref{GeneralEq}).
\end{corollary}

\begin{proof}
By Lemma \ref{VhatLin} and linearity we have
\begin{equation}\label{EqVhatLin}
	D\hat{V}(X)[Y]=\int_{M^n}\hat{\Xi}(X)\gb(Y,\nu(X))\dmu{}.
\end{equation}
It then follows from the form of $\hat{G}(X)$ in (\ref{GhatForm}) that under (\ref{GeneralEq})
\begin{align*}
\frac{\pa\hat{V}}{\pa t}=&D\hat{V}(X)\left[\frac{\pa X}{\pa t}\right]\\
=&D\hat{V}(X)\left[\hat{G}(X)\nu(X)\right]\\
=&\int_{M^n}\hat{\Xi}(X)\hat{G}(X)\dmu{}\\
=&0,
\end{align*}
thus under this weight function we have $\hat{V}(\Om_t)=\hat{V}(\Om_0)$ as long as the flow exists.
\end{proof}

\bibliographystyle{plain}

\end{document}